\documentclass[12pt]{article}
\usepackage{amsmath,amssymb,amsthm}
\numberwithin{equation}{section}
\usepackage{hyperref}
\usepackage{units}
\usepackage{color}
\usepackage[T1]{fontenc}
\usepackage[utf8]{inputenc}
\usepackage{authblk}
\usepackage{bm}
\usepackage{graphicx}

\usepackage{datetime}
\usepackage{color}

\usepackage[toc,page]{appendix}

\newcommand{\R}{{\mathbb R}}

\newtheorem{thm}{Theorem}[section]

\newtheorem{lemma}{Lemma}

\newtheorem{cor}[thm]{Corollary}

\title{Sums of arctangents and sums of products of arctangents \thanks{AMS Classification: 65B10}}

\author[]{Kunle Adegoke \thanks{adegoke00@gmail.com}}

\affil{Department of Physics and Engineering Physics, \mbox{Obafemi Awolowo University}, 220005 Ile-Ife, Nigeria}

\begin{document}

\date{}

\maketitle

\begin{abstract}
\noindent We present new infinite arctangent sums and infinite sums of products of arctangents. Many previously known evaluations appear as special cases of the general results derived in this paper.
\end{abstract}

\section{Introduction}
This paper reports the evaluation of infinite arctangent sums such as
\begin{equation}\label{equ.nekey4x}
\sum_{k = 1}^\infty  {\tan ^{ - 1} \frac{3}{{k^2  + 3k + 1 }}}  = \frac\pi 2
\end{equation}
and
\begin{equation}
\sum_{k = 1}^\infty  {\tan ^{ - 1} \frac{{4k}}{{(2k^2  - 1)^2 }}}  = \frac{\pi }{2}\,,
\end{equation}
and infinite sums of products of arctangents, such as
\begin{equation}\label{equ.mvf2qbz}
\sum_{k = 1}^\infty  {\tan ^{ - 1} \frac{2}{{k^2 }}\tan ^{ - 1} \frac{1}{k}}  = \frac{{\pi ^2 }}{8}
\end{equation}
and
\begin{equation}\label{equ.u0196cc}
\sum_{k = 1}^\infty  {\tan ^{ - 1} \frac{{6k}}{{k^4  - 7k^2  + 2}}\tan ^{ - 1} \frac{1}{{k^2  - k - 1}}} \tan ^{ - 1} \frac{1}{{k^2  + k - 1}} = \frac{{\pi ^3 }}{{64}}\,,
\end{equation}
and their generalizations.
The well-known identity
\[
\sum_{k = 1}^\infty  {\tan ^{ - 1} \frac{1}{{k^2  + k + 1 }}}  = \frac\pi 4
\]
and the presumably new identity~\eqref{equ.nekey4x} are both special cases of the following identity (section~\ref{sec.swgp99m}, identity~\eqref{equ.l11ss5w}):
\[
\sum_{k = 1}^\infty  {\tan ^{ - 1} \frac{{q}}{{k^2  + qk + 1 }}}  = \sum_{k = 1}^q {\tan ^{ - 1} \frac{1 }{k}}\,,
\]
being evaluations at $q=1$ and at $q=3$, respectively.

\bigskip
Identity~\eqref{equ.mvf2qbz} is obtained at $q=1$ from the following identity (identity~\eqref{equ.vs31gxz} of section~\ref{sec.iga29qz}):
\[
\begin{split}
&\sum_{k = 1}^\infty  {\tan ^{ - 1} \frac{{2q^2 }}{{(k + q - 1)^2 }}\tan ^{ - 1} \frac{q}{{k + q - 1}}}\\
&\qquad = \frac{{\pi ^2 }}{8} + \sum_{k = 2}^q {\tan ^{ - 1} \frac{q}{{k - 1}}\tan ^{ - 1} \frac{q}{{k + q - 1}}}\,.
\end{split}
\]
Identity~\eqref{equ.u0196cc} comes from identity~\eqref{equ.aixbx8z} of section~\ref{sec.ewieaba}, namely,
\[
\sum_{k = 1}^\infty  {\tan ^{ - 1} \frac{{6\alpha k}}{{k^4  -7k^2  + \alpha ^2  + 1}}\tan ^{ - 1} \frac{\alpha }{{k^2  - k - 1}}\tan ^{ - 1} \frac{\alpha }{{k^2  + k - 1}}}  = \left( {\tan ^{ - 1} \alpha } \right)^3\,.
\]
\begin{lemma}\label{L1}
Let $\alpha$ be a real number, let $m$ and $q$ be positive integers and let $\{ f(k)\} _{k = 1}^\infty  $ be a real positive non-decreasing sequence such that \mbox{$\lim_{k\to\infty}f(k)=\infty$}, then 
\[
\begin{split}
&\sum_{k = 1}^\infty  {\tan ^{ - 1} \frac{{\alpha (f(k + mq) - f(k))}}{{f(k)f(k + mq) + \alpha ^2 }}\prod_{j = 1}^{m - 1} {\tan ^{ - 1} \frac{\alpha }{{f(k + jq)}}} }\\
&\qquad\qquad= \sum_{k = 1}^q {\prod_{j = 0}^{m - 1} {\tan ^{ - 1} \frac{\alpha }{{f(k + jq)}}} }\,.
\end{split}
\]

\end{lemma}
\begin{lemma}\label{LP}
Let $\alpha$ be a real number, let $m$ and $q$ be positive integers and let $\{ f(k)\} _{k = 1}^\infty  $ be a real positive non-decreasing sequence such that \mbox{$\lim_{k\to\infty}f(k)=\infty$}, then 
\[
\begin{split}
&\sum_{k = 1}^\infty  {\tan ^{ - 1} \frac{{\alpha (f(k + mq) - f(k))}}{{f(k)f(k + mq) + \alpha ^2 }}\tan ^{ - 1} \frac{{\alpha (f(k + mq) + f(k))}}{{f(k)f(k + mq) - \alpha ^2 }}\prod_{j = 1}^{m - 1} {\left( {\tan ^{ - 1} \frac{\alpha }{{f(k + jq)}}} \right)^2 } }\\
&\qquad\qquad= \sum_{k = 1}^q {\prod_{j = 0}^{m - 1} {\left( {\tan ^{ - 1} \frac{\alpha }{{f(k + jq)}}} \right)^2 } }\,.
\end{split}
\]

\end{lemma}
Lemma~\ref{L1} and Lemma~\ref{LP} follow directly from identity~(2.1) of~\cite{adegoke} and the trigonometric identities
\begin{equation}\label{equ.ie9kz2w}
\tan ^{ - 1} \frac{\lambda }{x} - \tan ^{ - 1} \frac{\lambda }{y} = \tan ^{ - 1} \frac{{\lambda (y - x)}}{{xy + \lambda ^2 }},\quad\mbox{$\lambda^2/xy\ge -1$}\,,
\end{equation}
and
\begin{equation}\label{equ.zkvzqba}
\tan ^{ - 1} \frac{\lambda }{x} + \tan ^{ - 1} \frac{\lambda }{y} = \tan ^{ - 1} \frac{{\lambda (y + x)}}{{xy - \lambda ^2 }}\,,\quad\mbox{$\lambda^2/xy< 1$}\,.
\end{equation}
Throughout this paper, the principal value of the arctangent function is assumed.

\bigskip

We shall adopt the following conventions for empty sums and empty products:
\[
\sum_{k = 1}^0 {f(k)}  = 0,\quad\prod_{k = 1}^0 {f(k)}  = 1\,.
\]
\section{Main Results}
\begin{thm}\label{thm.sveqk2u}
If $\alpha$ and $\beta\ge -1$ are real numbers and $m$ and $q$ are positive integers, then
\[
\begin{split}
&\sum_{k = 1}^\infty  {\tan ^{ - 1} \frac{{\alpha mq}}{{k^2  + (2\beta + mq)k + \beta(\beta+mq) + \alpha ^2 }}\prod_{j = 1}^{m - 1} {\tan ^{ - 1} \frac{\alpha }{{k + jq + \beta}}} }\\
&\qquad\qquad\qquad= \sum_{k = 1}^q {\prod_{j = 0}^{m - 1} {\tan ^{ - 1} \frac{\alpha }{{k + jq + \beta}}} }\,.
\end{split}
\]

\end{thm}
In particular,
\begin{equation}\label{equ.hh4gkqe}
\sum_{k = 1}^\infty  {\tan ^{ - 1} \frac{{\alpha q}}{{k^2  + (2\beta  + q)k + q\beta  + \alpha ^2  + \beta ^2 }}}  = \sum_{k = 1}^q {\tan ^{ - 1} \frac{\alpha }{{k + \beta }}}\,,
\end{equation}
\begin{equation}\label{equ.pstkj0j}
\begin{split}
&\sum_{k = 1}^\infty  {\tan ^{ - 1} \frac{{2\alpha q}}{{k^2  + 2(\beta  + q)k + 2q\beta  + \alpha ^2  + \beta ^2 }}\tan ^{ - 1} \frac{\alpha }{{k + q + \beta }}}\\ 
&\qquad= \sum_{k = 1}^q {\tan ^{ - 1} \frac{\alpha }{{k + \beta }}\tan ^{ - 1} \frac{\alpha }{{k + q + \beta }}} 
\end{split}
\end{equation}
and
\begin{equation}\label{equ.swkgzwc}
\begin{split}
&\sum_{k = 1}^\infty  {\tan ^{ - 1} \frac{{3\alpha q}}{{k^2  + (2\beta  + 3q)k + 3q\beta  + \alpha ^2  + \beta ^2 }}\tan ^{ - 1} \frac{\alpha }{{k + q + \beta }}} \tan ^{ - 1} \frac{\alpha }{{k + 2q + \beta }}\\
&\qquad= \sum_{k = 1}^q {\tan ^{ - 1} \frac{\alpha }{{k + \beta }}\tan ^{ - 1} \frac{\alpha }{{k + q + \beta }}} \tan ^{ - 1} \frac{\alpha }{{k + 2q + \beta }}\,.
\end{split}
\end{equation}
\begin{proof}
Use $f(k)=k+\beta$ in Lemma~\ref{L1}.
\end{proof}
\begin{cor}\label{thm.pori4ri}
If $q$ and $m$ are positive integers and $\beta\ge -1$ is a real number, then
\[
\sum_{k = 1}^\infty  {\tan ^{ - 1} \frac{{2m^2 q^2 }}{{(2k + 2\beta + mq)^2 }}\prod_{j = 1}^{m - 1} {\tan ^{ - 1} \frac{{mq}}{{2(k + \beta + jq)}}} }  = \sum_{k = 1}^q {\prod_{j = 0}^{m - 1} {\tan ^{ - 1} \frac{{mq}}{{2(k + \beta + jq)}}} }\,. 
\]
\end{cor}
In particular,
\begin{equation}\label{equ.ah4yf0f}
\sum_{k = 1}^\infty  {\tan ^{ - 1} \frac{{2q^2 }}{{(2k + 2\beta + q)^2 }}}  = \sum_{k = 1}^q {\tan ^{ - 1} \frac{q}{{2(k + \beta)}}}\,.
\end{equation}
\begin{proof}
Set $\alpha=mq/2$ in the identity of Theorem~\ref{thm.sveqk2u}.
\end{proof}
\begin{thm}\label{thm.slsjsoq}
If $\alpha$ and $\beta$ are real numbers and $m$ and $q$ are positive integers, then
\[
\begin{split}
&\sum_{k = 1}^\infty  {\tan ^{ - 1} \frac{{2\alpha mqk}}{{k^4  - (m^2 q^2 - 2\beta)k^2  + \beta ^2  + \alpha ^2 }}\prod_{j = 1}^{m - 1} {\tan ^{ - 1} \frac{\alpha }{{k^2  + (2j-m)qk -q^2j(m-j)  + \beta }}} } \\
&\qquad= \sum_{k = 1}^q {\prod_{j = 0}^{m - 1} {\tan ^{ - 1} \frac{\alpha }{{k^2  + (2j-m)qk - q^2j(m-j)  + \beta }}} }\,. 
\end{split}
\]

\end{thm}
In particular,
\begin{equation}\label{equ.fl6y4rm}
\sum_{k = 1}^\infty  {\tan ^{ - 1} \frac{{2\alpha qk}}{{k^4  + (2\beta  - q^2 )k^2  + \beta ^2  + \alpha ^2 }}}  = \sum_{k = 1}^q {\tan ^{ - 1} \frac{\alpha }{{k^2  - qk + \beta }}}\,.
\end{equation}
\begin{proof}
Use $f(k)=k^2-mqk+\beta$ in Lemma~\ref{L1}.
\end{proof}
\begin{thm}\label{thm.ivbxym1}
If $\alpha$ and $\beta$ are real numbers and $m$ and $q$ are positive integers, then
\[
\begin{split}
&\sum_{k = 1}^\infty  {\tan ^{ - 1} \frac{{\alpha mq}}{{k^2  + (2\beta  + mq)k + \beta (\beta  + mq) + \alpha ^2 }}\tan ^{ - 1} \frac{{\alpha (2k + 2\beta  + mq)}}{{k^2  + (2\beta  + mq)k + \beta (\beta  + mq) - \alpha ^2 }} \times }
 \\
&\qquad\qquad \times \prod_{j = 1}^{m - 1} {\left( {\tan ^{ - 1} \frac{\alpha }{{k + jq + \beta }}} \right)^2 }  = \sum_{k = 1}^q {\prod_{j = 0}^{m - 1} {\left( {\tan ^{ - 1} \frac{\alpha }{{k + jq + \beta }}} \right)^2 } }\,. 
\end{split}
\]

\end{thm}
\begin{proof}
Use $f(k)=k+\beta$ in Lemma~\ref{LP}.
\end{proof}
In particular,
\begin{equation}\label{equ.ljah2jn}
\begin{split}
&\sum_{k = 1}^\infty  {\tan ^{ - 1} \frac{{\alpha q}}{{k^2  + (2\beta  + q)k + \beta (\beta  + q) + \alpha ^2 }}\tan ^{ - 1} \frac{{\alpha (2k + 2\beta  + q)}}{{k^2  + (2\beta  + q)k + \beta (\beta  + q) - \alpha ^2 }}}\\
&\qquad\qquad= \sum_{k = 1}^q {\left( {\tan ^{ - 1} \frac{\alpha }{{k + jq + \beta }}} \right)^2 }\,.
\end{split}
\end{equation}
\begin{thm}\label{thm.lqwviov}
If $\alpha$ and $\beta$ are real numbers and $m$ and $q$ are positive integers, then
\[
\begin{split}
&\sum_{k = 1}^\infty  {\tan ^{ - 1} \frac{{2\alpha mqk}}{{k^4  + (2\beta  - m^2 q^2 )k^2  + \alpha ^2  + \beta ^2 }}\tan ^{ - 1} \frac{{2\alpha (k^2  + \beta )}}{{k^4  + (2\beta  - m^2 q^2 )k^2  - \alpha ^2  + \beta ^2 }}}  \times\\
&\qquad\times \prod_{j = 1}^{m - 1} {\left( {\tan ^{ - 1} \frac{\alpha }{{k^2  - qk(m - 2j) - q^2 j(m - j) + \beta }}} \right)^2 }\\
&\qquad\qquad = \sum_{k = 1}^q {\prod_{j = 0}^{m - 1} {\left( {\tan ^{ - 1} \frac{\alpha }{{k^2  - qk(m - 2j) - q^2 j(m - j) + \beta }}} \right)^2 } }\,.
\end{split}
\]

\end{thm}
\begin{proof}
Use $f(k)=k^2-mqk+\beta$ in Lemma~\ref{LP}.
\end{proof}
\begin{cor}
If $m$ and $q$ are positive integers, then
\[
\begin{split}
&\sum_{k = 1}^\infty  {\tan ^{ - 1} \frac{{2m^2 qk}}{{2k^4  - (mq^2  - 1)2mk^2  + m^2 }}\tan ^{ - 1} \frac{{m(2k^2  + m)}}{{2k^2 (k^2  + m - m^2 q^2 )}}}  \times \\
&\qquad\times \prod_{j = 1}^{m - 1} {\left( {\tan ^{ - 1} \frac{m}{{2k^2  - 2qk(m - 2j) - 2q^2 j(m - j) + m}}} \right)^2 }\\ 
&\qquad\qquad= \sum_{k = 1}^q {\prod_{j = 0}^{m - 1} {\left( {\tan ^{ - 1} \frac{m}{{2k^2  - 2qk(m - 2j) - 2q^2 j(m - j) + m}}} \right)^2 } }\,. 
\end{split}
\]
\end{cor}
\section{Examples}
\subsection{Sums of arctangents}
\subsubsection{Examples from identity~\eqref{equ.hh4gkqe}}\label{sec.swgp99m}
Evaluating~\eqref{equ.hh4gkqe} at $\beta=-1$ gives
\begin{equation}\label{equ.tcbtsqt}
\sum_{k = 1}^\infty  {\tan ^{ - 1} \frac{{\alpha q}}{{k^2  + (q - 2)(k - 1) + \alpha ^2  - 1}}}  = \frac{\pi }{2} + \sum_{k = 2}^q {\tan ^{ - 1} \frac{\alpha }{{k - 1}}},\quad\alpha\in\R^+\,,
\end{equation}
from which we obtain
\begin{equation}\label{equ.iefhc2p}
\sum_{k = 1}^\infty  {\tan ^{ - 1} \frac{{q}}{{k^2  + (q - 2)(k - 1)}}}  = \frac{\pi }{2} + \sum_{k = 2}^q {\tan ^{ - 1} \frac{1 }{{k - 1}}}\,,
\end{equation}
\begin{equation}\label{equ.jjrqfmc}
\sum_{k = 1}^\infty  {\tan ^{ - 1} \frac{{2\alpha }}{{k^2  + \alpha ^2  - 1}}}  = \frac{\pi }{2} + \tan ^{ - 1} \alpha
\end{equation}
and
\begin{equation}\label{equ.kgx4gck}
\sum_{k = 1}^\infty  {\tan ^{ - 1} \frac{\alpha }{{k^2  - k + \alpha ^2 }}}  = \frac{\pi }{2}\,.
\end{equation}
The identity~\eqref{equ.kgx4gck} is valid for any positive real $\alpha$, and in particular evaluation at $\alpha=1$ gives
\begin{equation}\label{equ.macv3oy}
\sum_{k = 1}^\infty  {\tan ^{ - 1} \frac{1 }{{k^2  - k + 1 }}}  = \frac{\pi }{2}\,.
\end{equation}

Evaluating~\eqref{equ.jjrqfmc} at $\alpha=1$ gives the well known result,
\begin{equation}
\sum_{k = 1}^\infty  {\tan ^{ - 1} \frac{2}{{k^2 }}}  = \frac{{3\pi }}{4}\,.
\end{equation}
Putting $\beta=0$ in identity~\eqref{equ.hh4gkqe} provides the evaluation
\begin{equation}\label{equ.s0rb2zk}
\sum_{k = 1}^\infty  {\tan ^{ - 1} \frac{{\alpha q}}{{k^2  + qk + \alpha ^2 }}}  = \sum_{k = 1}^q {\tan ^{ - 1} \frac{\alpha }{k}}\,,\quad\alpha\in\R\,,
\end{equation}
from which we get
\begin{equation}\label{equ.l11ss5w}
\sum_{k = 1}^\infty  {\tan ^{ - 1} \frac{{q}}{{k^2  + qk + 1 }}}  = \sum_{k = 1}^q {\tan ^{ - 1} \frac{1 }{k}}
\end{equation}
and
\begin{equation}\label{equ.s24gqww}
\sum_{k = 1}^\infty  {\tan ^{ - 1} \frac{{\alpha}}{{k^2  + k + \alpha ^2 }}}  = \tan ^{ - 1}\alpha\,.
\end{equation}
Setting $\alpha=\tan\theta$ in identity~\eqref{equ.s24gqww} gives
\begin{equation}
\theta  = \sum_{k = 1}^\infty  {\tan ^{ - 1} \frac{{\tan \theta }}{{k^2  + k + \tan ^2 \theta }}}\,,
\end{equation}
providing an infinite arctangent sum expansion for any angle $\theta$, such that $-\pi/2<\theta<\pi/2$.

\bigskip

Upon setting $\beta=-1/2$ in identity~\eqref{equ.hh4gkqe} we obtain
\begin{equation}
\sum_{k = 1}^\infty  {\tan ^{ - 1} \frac{{4\alpha q}}{{4k^2  + 4k(q - 1) - 2q + 1 + 4\alpha ^2 }}}  = \sum_{k = 1}^q {\tan ^{ - 1} \frac{{2\alpha }}{{2k - 1}}}\,,\quad\alpha\in\R\,,
\end{equation}
from which we also obtain, at $\alpha=q/2$,
\begin{equation}
\sum_{k = 1}^\infty  {\tan ^{ - 1} \frac{{2q^2 }}{{(2k + q - 1)^2 }}}  = \sum_{k = 1}^q {\tan ^{ - 1} \frac{q}{{2k - 1}}}\,,
\end{equation}
with the special value
\begin{equation}
\sum_{k = 1}^\infty  {\tan ^{ - 1} \frac1{2k^2 }}  = \frac\pi 4\,.
\end{equation}
Choosing $\alpha=\sin\theta$, $\beta=\cos\theta$ and $q=1$ in identity~\eqref{equ.hh4gkqe} gives
\begin{equation}
\theta  = 2\sum_{k = 1}^\infty  {\tan ^{ - 1} \frac{{\sin \theta }}{{k^2  + (2\cos \theta  + 1)k + \cos \theta  + 1}}}\,,
\end{equation}
thereby providing infinite arctangent sums for angles $-\pi/2\le\theta\le\pi/2$. In particular,
\begin{equation}
\frac{\pi }{6} = \sum_{k = 1}^\infty  {\tan ^{ - 1} \frac{{\sqrt 3 }}{{2k^2  + 4k + 3}}}\,.
\end{equation}
Replacing $\beta$ with $\beta-1$ in identity~\eqref{equ.hh4gkqe} and evaluating at $q=1$ gives
\begin{equation}
\sum_{k = 1}^\infty  {\tan ^{ - 1} \frac{\alpha }{{k^2  + (2\beta  - 1)k - \beta  + \alpha ^2  + \beta ^2 }}}  = \tan ^{ - 1} \frac{\alpha }{\beta }\,,
\end{equation}
which is equivalent to identity~(3) of~\cite{glasser}.
\subsubsection{Examples from identity~\eqref{equ.ah4yf0f}}
Evaluating identity~\eqref{equ.ah4yf0f} at $\beta=0$ and at $\beta=-1$, respectively, we have
\begin{equation}
\sum_{k = 1}^\infty  {\tan ^{ - 1} \frac{{2q^2 }}{{(2k + q)^2 }}}  = \sum_{k = 1}^q {\tan ^{ - 1} \frac{q}{{2k}}}
\end{equation}
and
\begin{equation}
\sum_{k = 1}^\infty  {\tan ^{ - 1} \frac{{2q^2 }}{{(2k + q - 2)^2 }}}  = \frac{\pi }{2} + \sum_{k = 2}^q {\tan ^{ - 1} \frac{q}{{2(k - 1)}}}\,
\end{equation}
of which a special value is
\begin{equation}
\sum_{k = 1}^\infty  {\tan ^{ - 1} \frac{2}{{(2k - 1)^2 }}}  = \frac{\pi }{2}\,.
\end{equation}

\subsubsection{Examples from identity~\eqref{equ.fl6y4rm}}
Putting $\alpha=1=\beta$ in identity~\eqref{equ.fl6y4rm} we obtain
\begin{equation}
\sum_{k = 1}^\infty  {\tan ^{ - 1} \frac{{2qk}}{{k^4  + (2 - q^2 )k^2  + 2}}}  = \sum_{k = 1}^q {\tan ^{ - 1} \frac{1}{{k^2  - qk + 1}}}\,,
\end{equation}
which at $q=1$ gives
\begin{equation}
\sum_{k = 1}^\infty  {\tan ^{ - 1} \frac{{2k}}{{k^4  + k^2  + 2}}}  = \frac{\pi }{4}\,,
\end{equation}
and at $q=2$ gives
\begin{equation}
\sum_{k = 1}^\infty  {\tan ^{ - 1} \frac{{4k}}{{k^4  - 2k^2  + 2}}}  = \frac{{3\pi }}{4}\,,
\end{equation}
a result that was also reported in~\cite{bragg}.

\bigskip

At $\beta=2$, $q=3$, identity~\eqref{equ.fl6y4rm} gives
\begin{equation}
\sum_{k = 1}^\infty  {\tan ^{ - 1} \frac{{6\alpha k}}{{k^4  - 5k^2  + 4 + \alpha^2}}}  = \pi + \tan^{-1}\frac\alpha 2\,,\quad\alpha\in\R^+\,,
\end{equation}
which at $\alpha=2$ produces
\begin{equation}\label{equ.q0xt5i4}
\sum_{k = 1}^\infty  {\tan ^{ - 1} \frac{{12k}}{{k^4  - 5k^2  + 8}}}  = \frac{5\pi }{4}\,.
\end{equation}
Evaluating identity~\eqref{equ.fl6y4rm} at $\beta=0$, $q=1$ yields
\begin{equation}
\sum_{k = 1}^\infty  {\tan ^{ - 1} \frac{{2\alpha k}}{{k^4  - k^2  + \alpha ^2 }}}  = \frac{\pi }{2}\,,
\end{equation}
which at $\alpha=1/2$ gives the special value
\[
\sum_{k = 1}^\infty  {\tan ^{ - 1} \frac{{4k}}{{(2k^2  - 1)^2 }}}  = \frac{\pi }{2}\,.
\]
Setting $\beta=q^2/2$ in identity~\eqref{equ.fl6y4rm} gives
\begin{equation}\label{equ.du2vgo4}
\sum_{k = 1}^\infty  {\tan ^{ - 1} \frac{{8\alpha qk}}{{4k^4  + 4\alpha ^2  + q^4 }}}  = \sum_{k = 1}^q {\tan ^{ - 1} \frac{{2\alpha }}{{2k^2  - 2qk + q^2 }}}\,,
\end{equation}
and, in particular, with $\alpha=q^2/2$, we have
\begin{equation}
\sum_{k = 1}^\infty  {\tan ^{ - 1} \frac{{2q^3 k}}{{2k^4  + q^4 }}}  = \sum_{k = 1}^q {\tan ^{ - 1} \frac{{q^2 }}{{2k^2  - 2qk + q^2 }}}\,,
\end{equation}
giving at $q=1$, the special value
\begin{equation}
\sum_{k = 1}^\infty  {\tan ^{ - 1} \frac{{2k}}{{2k^4  + 1}}}  = \frac{\pi }{4}\,.
\end{equation}
Evaluating identity~\eqref{equ.du2vgo4} at $q=2$ gives
\begin{equation}
\sum_{k = 1}^\infty  {\tan ^{ - 1} \frac{{4\alpha k}}{{k^4  + \alpha ^2  + 4}}}  = \tan ^{ - 1} \frac{\alpha }{2} + \tan ^{ - 1} \alpha\,,
\end{equation}
which was also reported in~\cite{boros}.

\bigskip

Setting $2\alpha=\tan\theta$ and $q=1$ in identity~\eqref{equ.du2vgo4} gives
\begin{equation}
\theta  = \sum_{k = 1}^\infty  \tan^{-1}{\frac{{4k\tan \theta }}{{4k^4  + \sec ^2 \theta}}}\,,
\end{equation}
providing an infinite arctangent sum expansion for any angle $\theta$, such that $-\pi/2<\theta<\pi/2$. In particular,
\begin{equation}
\frac{\pi }{3} = \sum_{k = 1}^\infty  {\tan ^{ - 1} \frac{{k\sqrt 3 }}{{k^4  + 1}}}\,,\quad\frac{\pi }{6} = \sum_{k = 1}^\infty  {\tan ^{ - 1} \frac{{k\sqrt 3 }}{{3k^4  + 1}}}\,. 
\end{equation}
\subsection{Sums of products of arctangents}
\subsubsection{Examples from Theorem~\ref{thm.sveqk2u}}
Using $\beta=-1$ in the identity of Theorem~\ref{thm.sveqk2u}, we have
\begin{equation}\label{equ.bdchdd}
\begin{split}
&\sum_{k = 1}^\infty  {\tan ^{ - 1} \frac{{\alpha mq}}{{k^2  + (mq - 2)(k - 1) + \alpha ^2  - 1}}\prod_{j = 1}^{m - 1} {\tan ^{ - 1} \frac{\alpha }{{k + jq - 1}}} }\\
 &\qquad= \frac{\pi }{2}\prod_{j = 1}^{m - 1} {\tan ^{ - 1} \frac{\alpha }{{jq}}}  + \sum_{k = 2}^q {\prod_{j = 0}^{m - 1} {\tan ^{ - 1} \frac{\alpha }{{k + jq - 1}}} },\quad\alpha\in\R^+\, 
\end{split}
\end{equation}
In particular,
\begin{equation}
\begin{split}
&\sum_{k = 1}^\infty  {\tan ^{ - 1} \frac{{\alpha m}}{{k^2  + (m - 2)(k - 1) + \alpha ^2  - 1}}\prod_{j = 1}^{m - 1} {\tan ^{ - 1} \frac{\alpha }{{k + j - 1}}} }\\ 
 &\qquad\qquad= \frac{\pi }{2}\prod_{j = 1}^{m - 1} {\tan ^{ - 1} \frac{\alpha }{j}}\,,
\end{split}
\end{equation}
which at $m=2$ gives
\begin{equation}\label{equ.rix1u3f}
\sum_{k = 1}^\infty  {\tan ^{ - 1} \frac{{2\alpha }}{{k^2  + \alpha ^2  - 1}}\tan ^{ - 1} \frac{\alpha }{k}}  = \frac{\pi }{2}\tan ^{ - 1} \alpha,\quad\alpha\in\R^+\,,
\end{equation}
yielding at $\alpha=1$, $\alpha=\sqrt 3$ and $\alpha=1/\sqrt{3}$, respectively,
\[
\sum_{k = 1}^\infty  {\tan ^{ - 1} \frac{2}{{k^2 }}\tan ^{ - 1} \frac{1}{k}}  = \frac{{\pi ^2 }}{8}\,,
\]
\begin{equation}\label{equ.q1fro1u}
\sum_{k = 1}^\infty  {\tan ^{ - 1} \frac{{2\sqrt 3 }}{{k^2  + 2}}\tan ^{ - 1} \frac{{\sqrt 3 }}{k}}  = \frac{{\pi ^2 }}{6}
\end{equation}
and
\begin{equation}\label{equ.hn7l2vs}
\sum_{k = 1}^\infty  {\tan ^{ - 1} \frac{{2\sqrt 3 }}{{3k^2  - 2}}\tan ^{ - 1} \frac{{\sqrt 3 }}{{3k}}}  = \frac{{\pi ^2 }}{{12}}\,.
\end{equation}
At $m=3$, the identity~\eqref{equ.bdchdd} gives
\begin{equation}
\sum_{k = 1}^\infty  {\tan ^{ - 1} \frac{{3\alpha }}{{k^2  + k + \alpha ^2  - 2}}\tan ^{ - 1} \frac{\alpha }{k}} \tan ^{ - 1} \frac{\alpha }{{k + 1}} = \frac{\pi }{2}\tan ^{ - 1} \alpha \tan ^{ - 1} \frac{\alpha }{2}\,,
\end{equation}
from which we get the evaluations
\begin{equation}
\sum_{k = 1}^\infty  {\tan ^{ - 1} \frac{{3\sqrt 2 }}{{k^2  + k}}\tan ^{ - 1} \frac{{\sqrt 2 }}{k}} \tan ^{ - 1} \frac{{\sqrt 2 }}{{k + 1}} = \frac{\pi }{2}\tan ^{ - 1} \sqrt 2 \tan ^{ - 1} \frac{{\sqrt 2 }}{2}
\end{equation}
and
\begin{equation}
\sum_{k = 1}^\infty  {\tan ^{ - 1} \frac{{3\sqrt 3 }}{{k^2  + k + 1}}\tan ^{ - 1} \frac{{\sqrt 3 }}{k}} \tan ^{ - 1} \frac{{\sqrt 3 }}{{k + 1}} = \frac{{\pi ^2 }}{6}\tan ^{ - 1} \frac{{\sqrt 3 }}{2}\,,
\end{equation}
at $\alpha=\sqrt 2$ and $\alpha=\sqrt 3$, respectively.
The following four-member sum
\begin{equation}
\sum_{k = 1}^\infty  {\tan ^{ - 1} \frac{{4\sqrt 3 }}{{k^2  + 2k}}\tan ^{ - 1} \frac{{\sqrt 3 }}{k}} \tan ^{ - 1} \frac{{\sqrt 3 }}{{k + 1}}\tan ^{ - 1} \frac{{\sqrt 3 }}{{k + 2}} = \frac{{\pi ^3 }}{{36}}\tan ^{ - 1} \frac{{\sqrt 3 }}{2}\,,
\end{equation}
is produced at $m=4$ and $\alpha=\sqrt 3$ in identity~\eqref{equ.bdchdd}.
\subsubsection{Examples from Corollary~\ref{thm.pori4ri}}\label{sec.iga29qz}
If we set $\beta=-1/2$ in the identity of Corollary~\ref{thm.pori4ri} we have
\begin{equation}
\sum_{k = 1}^\infty  {\tan ^{ - 1} \frac{{2m^2 q^2 }}{{(2k + mq - 1)^2 }}\prod_{j = 1}^{m - 1} {\tan ^{ - 1} \frac{{mq}}{{2k + 2jq - 1}}} }  = \sum_{k = 1}^q {\prod_{j = 0}^{m - 1} {\tan ^{ - 1} \frac{{mq}}{{2k + 2jq - 1}}} }\,,
\end{equation}
while at $\beta=-1$ we have
\begin{equation}
\sum_{k = 1}^\infty  {\tan ^{ - 1} \frac{{2m^2 q^2 }}{{(2k + mq - 2)^2 }}\prod_{j = 1}^{m - 1} {\tan ^{ - 1} \frac{{mq}}{{2(k + jq - 1)}}} }  = \sum_{k = 1}^q {\prod_{j = 0}^{m - 1} {\tan ^{ - 1} \frac{{mq}}{{2(k + jq - 1)}}} }\,,
\end{equation}
which at $m=2$ gives
\begin{equation}\label{equ.vs31gxz}
\begin{split}
&\sum_{k = 1}^\infty  {\tan ^{ - 1} \frac{{2q^2 }}{{(k + q - 1)^2 }}\tan ^{ - 1} \frac{q}{{k + q - 1}}}\\
&\qquad = \frac{{\pi ^2 }}{8} + \sum_{k = 2}^q {\tan ^{ - 1} \frac{q}{{k - 1}}\tan ^{ - 1} \frac{q}{{k + q - 1}}}\,,
\end{split}
\end{equation}
which at $q=1$ gives identity~\eqref{equ.mvf2qbz}. 
\subsubsection{Examples from Theorem~\ref{thm.slsjsoq}}\label{sec.ewieaba}
Upon setting $q=1$ in~Theorem~\ref{thm.slsjsoq} we obtain
\begin{equation}\label{equ.k733t9v}
\begin{split}
&\sum_{k = 1}^\infty  {\tan ^{ - 1} \frac{{2\alpha mk}}{{k^4  + (2\beta  - m^2 )k^2  + \beta ^2  + \alpha ^2 }}\prod_{j = 1}^{m - 1} {\tan ^{ - 1} \frac{\alpha }{{k^2  - (m - 2j)k + j^2  + \beta  - mj}}} }\\ 
 &\qquad\qquad= \prod_{j = 0}^{m - 1} {\tan ^{ - 1} \frac{\alpha }{{j^2  + 2j + \beta  - mj - m + 1}}},\quad\alpha,\beta\in\R\,.
\end{split}
\end{equation}
In particular, we have
\begin{equation}\label{equ.hf1f5ua}
\begin{split}
&\sum_{k = 1}^\infty  {\tan ^{ - 1} \frac{{4\alpha k}}{{k^4  + 2(\beta  - 2)k^2  + \alpha ^2  + \beta ^2 }}\tan ^{ - 1} \frac{\alpha }{{k^2  + \beta  - 1}}}\\
&\qquad\qquad= \tan ^{ - 1} \frac{\alpha }{{\beta  - 1}}\tan ^{ - 1} \frac{\alpha }{\beta }
\end{split}
\end{equation}
and
\begin{equation}\label{equ.nm9vnkw}
\begin{split}
&\sum_{k = 1}^\infty  {\tan ^{ - 1} \frac{{6\alpha k}}{{k^4  + (2\beta  - 9)k^2  + \alpha ^2  + \beta ^2 }}\tan ^{ - 1} \frac{\alpha }{{k^2  - k + \beta  - 2}}\tan ^{ - 1} \frac{\alpha }{{k^2  + k + \beta  - 2}}} \\
&\qquad\qquad= \left( {\tan ^{ - 1} \frac{\alpha }{{\beta  - 2}}} \right)^2 \tan ^{ - 1} \frac{\alpha }{\beta }\,,
\end{split}
\end{equation}
which at $\beta=1$ give
\begin{equation}\label{equ.hyiedj3}
\sum_{k = 1}^\infty  {\tan ^{ - 1} \frac{{4\alpha k}}{{k^4  - 2k^2  + \alpha ^2  + 1}}\tan ^{ - 1} \frac{\alpha }{{k^2 }}}  = \frac{\pi }{2}\tan ^{ - 1} \alpha
\end{equation}
and
\begin{equation}\label{equ.aixbx8z}
\sum_{k = 1}^\infty  {\tan ^{ - 1} \frac{{6\alpha k}}{{k^4  -7k^2  + \alpha ^2  + 1}}\tan ^{ - 1} \frac{\alpha }{{k^2  - k - 1}}\tan ^{ - 1} \frac{\alpha }{{k^2  + k - 1}}}  = \left( {\tan ^{ - 1} \alpha } \right)^3\,,
\end{equation}
producing at $\alpha=1$,
\begin{equation}
\sum_{k = 1}^\infty  {\tan ^{ - 1} \frac{{4k}}{{k^4  - 2k^2  + 2}}\tan ^{ - 1} \frac{1}{{k^2 }}}  = \frac{{\pi ^2 }}{8}
\end{equation}
and
\[
\sum_{k = 1}^\infty  {\tan ^{ - 1} \frac{{6k}}{{k^4  - 7k^2  + 2}}\tan ^{ - 1} \frac{1}{{k^2  - k - 1}}} \tan ^{ - 1} \frac{1}{{k^2  + k - 1}} = \frac{{\pi ^3 }}{{64}}\,.
\]
Special values from~\eqref{equ.hyiedj3} include
\begin{equation}
\sum_{k = 1}^\infty  {\tan ^{ - 1} \frac{{4k\sqrt 3 }}{{k^4  - 2k^2  + 4}}\tan ^{ - 1} \frac{{\sqrt 3 }}{{k^2 }}}  = \frac{{\pi ^2 }}{6}
\end{equation}
and
\begin{equation}
\sum_{k = 1}^\infty  {\tan ^{ - 1} \frac{{4k\sqrt 3 }}{{3k^4  - 6k^2  + 4}}\tan ^{ - 1} \frac{{\sqrt 3 }}{{3k^2 }}}  = \frac{{\pi ^2 }}{{12}}\,.
\end{equation}
By first replacing $\beta$ with $\beta-1$ and then setting $\alpha=\sin\theta$ and $\beta=\cos\theta$, the identity~\eqref{equ.hf1f5ua} can be put in the form,
\begin{equation}
\frac{{\theta ^2 }}{2} = \sum_{k = 1}^\infty  {\tan ^{ - 1} \frac{{4k\sin \theta }}{{k^4  - 4k^2 \sin ^2 (\theta /2) + 4\cos ^2 (\theta /2)}}\tan ^{ - 1} \frac{{\sin \theta }}{{k^2  + \cos \theta }}}\,,
\end{equation}
suitable for expressing the square of any angle $\theta$ with magnitude less that $\pi/2$ as an infinite sum of products of two arctangents. In particular, at $\theta=\pi/3$, we have
\begin{equation}
\frac{{\pi ^2 }}{{18}} = \sum_{k = 1}^\infty  {\tan ^{ - 1} \frac{{2k\sqrt 3 }}{{k^4  - k^2  + 3}}\tan ^{ - 1} \frac{{\sqrt 3 }}{{2k^2  + 1}}}\,.
\end{equation}
At $m=5$, $\alpha=1=\beta=q$ in identity~\eqref{equ.k733t9v} we have
\begin{equation}
\begin{split}
&\sum_{k = 1}^\infty  {\tan ^{ - 1} \frac{{10k}}{{k^4  - 23k^2  + 2}}\tan ^{ - 1} \frac{1}{{k^2  - 3k - 3}}} \tan ^{ - 1} \frac{1}{{k^2  - k - 5}}\times\\
&\qquad\qquad\times \tan ^{ - 1} \frac{1}{{k^2  - k + 5}}\tan ^{ - 1} \frac{1}{{k^2  + 3k - 3}}\\
&\qquad\qquad\qquad = \frac{\pi }{4}\left( {\tan ^{ - 1} \frac{1}{3}} \right)^2 \left( {\tan ^{ - 1} \frac{1}{5}} \right)^2\,. 
\end{split}
\end{equation}
\subsubsection{Examples from Theorem~\ref{thm.ivbxym1}}
Using $\beta=0$, $\alpha=1$, $q=1$ in identity~\eqref{equ.ljah2jn} gives
\begin{equation}
\sum_{k = 1}^\infty  {\tan ^{ - 1} \frac{1}{{k^2  + k + 1}}\tan ^{ - 1} \frac{{2k + 1}}{{k^2  + k - 1}}}  = \frac{{\pi ^2 }}{{16}}\,.
\end{equation}
Setting $q=1$ and $\alpha=1/2=-\beta$ in identity~\eqref{equ.ljah2jn}, we obtain
\begin{equation}
\sum_{k = 1}^\infty  {\tan ^{ - 1} \frac{1}{{2k^2 }}\tan ^{ - 1} \frac{{2k}}{{2k^2  - 1}}}  = \frac{{\pi ^2 }}{{16}}\,,
\end{equation}
while setting $q=2$ and $\alpha=1/2$, $\beta=-3/2$ in the same identity gives
\begin{equation}
\sum_{k = 1}^\infty  {\tan ^{ - 1} \frac{2}{{2k^2  - 2k - 1}}\tan ^{ - 1} \frac{{2k - 1}}{{2(k^2  - k - 1)}}}  = \frac{{\pi ^2 }}{8}\,.
\end{equation}
Using $m=2$, $q=1$ and $\alpha=1/2$, $\beta=-3/2$ in Theorem~\ref{thm.ivbxym1} gives
\begin{equation}
\sum_{k = 1}^\infty  {\left( {\tan ^{ - 1} \frac{1}{{2k - 1}}} \right)^2 \tan ^{ - 1} \frac{2}{{2k^2  - 2k - 1}}\tan ^{ - 1} \frac{{2k - 1}}{{2(k^2  - k - 1)}}}  = \frac{{\pi ^4 }}{{256}}\,.
\end{equation}
Using $m=3$, $q=1$ and $\alpha=1=\beta$ in Theorem~\ref{thm.ivbxym1} gives
\begin{equation}
\begin{split}
&\sum_{k = 1}^\infty  {\left( {\tan ^{ - 1} \frac{1}{{k + 2}}} \right)^2 \left( {\tan ^{ - 1} \frac{1}{{k + 3}}} \right)^2 \tan ^{ - 1} \frac{3}{{k^2  + 5k + 5}}\tan ^{ - 1} \frac{{2k + 5}}{{k^2  + 5k + 3}}}\\
&\qquad\qquad= \left( {\tan ^{ - 1} \frac{1}{2}} \right)^2 \left( {\tan ^{ - 1} \frac{1}{3}} \right)^2 \left( {\tan ^{ - 1} \frac{1}{4}} \right)^2\,.
\end{split}
\end{equation}
\subsubsection{Examples from Theorem~\ref{thm.lqwviov}}
Use of $m=1=q$ in the identity of Theorem~\ref{thm.lqwviov} gives
\begin{equation}
\begin{split}
&\sum_{k = 1}^\infty  {\tan ^{ - 1} \frac{{2\alpha k}}{{k^4  + (2\beta  - 1)k^2  + \beta ^2  + \alpha ^2 }}\tan ^{ - 1} \frac{{2\alpha (k^2  + \beta )}}{{k^4  + (2\beta  - 1)k^2  + \beta ^2  - \alpha ^2 }}}\\  &\qquad\qquad= \left( {\tan ^{ - 1} \frac{\alpha }{\beta }} \right)^2\,,
\end{split}
\end{equation}
giving, in particular,
\begin{equation}
\sum_{k = 1}^\infty  {\tan ^{ - 1} \frac{{8\alpha k}}{{4k^4  + 1 + 4\alpha ^2 }}\tan ^{ - 1} \frac{{4\alpha (2k^2  + 1)}}{{4k^4  + 1 - 4\alpha ^2 }}}  = \left( {\tan ^{ - 1} 2\alpha } \right)^2\,,
\end{equation}
yielding the evaluations
\begin{equation}
\sum_{k = 1}^\infty  {\tan ^{ - 1} \frac{{k\sqrt 3 }}{{k^4  + 1}}\tan ^{ - 1} \frac{{\sqrt 3 (2k^2  + 1)}}{{2k^4  - 1}}}  = \frac{{\pi ^2 }}{9}
\end{equation}
and
\begin{equation}
\sum_{k = 1}^\infty  {\tan ^{ - 1} \frac{{k\sqrt 3 }}{{3k^4  + 1}}\tan ^{ - 1} \frac{{\sqrt 3 (2k^2  + 1)}}{{6k^4  + 1}}}  = \frac{{\pi ^2 }}{{36}}\,,
\end{equation}
at $\alpha=\sqrt3/2$ and $\alpha=1/(2\sqrt3)$ respectively.

\bigskip

Setting $m=1$ and $\alpha=1/2=\beta$ in the identity of Theorem~\ref{thm.lqwviov} gives
\begin{equation}\label{equ.g3fwlld}
\begin{split}
&\sum_{k = 1}^\infty  {\tan ^{ - 1} \frac{{2qk}}{{2k^4  - (q^2  - 1)2k^2  + 1}}\tan ^{ - 1} \frac{{2k^2  + 1}}{{2k^2 (k^2  + 1 - q^2 )}}}\\
&\qquad= \sum_{k = 1}^q {\left( {\tan ^{ - 1} \frac{1}{{2k^2  - 2qk + 1}}} \right)^2 }\,,
\end{split}
\end{equation}
from which we get the special evaluations
\begin{equation}
\sum_{k = 1}^\infty  {\tan ^{ - 1} \frac{{2k}}{{2k^4  + 1}}\tan ^{ - 1} \frac{{2k^2  + 1}}{{2k^4 }}}  = \frac{{\pi ^2 }}{{16}}
\end{equation}
and
\begin{equation}
\sum_{k = 1}^\infty  {\tan ^{ - 1} \frac{{4k}}{{2k^4  - 6k^2  + 1}}\tan ^{ - 1} \frac{{2k^2  + 1}}{{2k^2 (k^2  - 3)}}}  = \frac{{\pi ^2 }}{8}\,.
\end{equation}
Similarly, $m=1$ and $\alpha=1/2=-\beta$ in the identity of Theorem~\ref{thm.lqwviov} gives
\begin{equation}\label{equ.mhe9kvi}
\begin{split}
&\sum_{k = 1}^\infty  {\tan ^{ - 1} \frac{{2qk}}{{2k^4  - (q^2  + 1)2k^2  + 1}}\tan ^{ - 1} \frac{{(2k^2  - 1)}}{{2k^2 (k^2  - 1 - q^2 )}}} \\
&\qquad = \sum_{k = 1}^q {\left( {\tan ^{ - 1} \frac{1}{{2k^2  - 2qk - 1}}} \right)^2 }\,,
\end{split}
\end{equation}
giving, in particular,
\begin{equation}
\sum_{k = 1}^\infty  {\tan ^{ - 1} \frac{{4k}}{{2k^4  - 4k^2  + 1}}\tan ^{ - 1} \frac{{(2k^2  - 1)}}{{2k^2 (k^2  - 2)}}}  = \frac{{\pi ^2 }}{{16}}\,.
\end{equation}
Setting $m=1=\alpha=\beta$ in the identity of Theorem~\ref{thm.lqwviov}, we have,
\begin{equation}
\begin{split}
&\sum_{k = 1}^\infty  {\tan ^{ - 1} \frac{{2kq}}{{k^4  - (q^2  - 2)k^2  + 2}}\tan ^{ - 1} \frac{{2(k^2  + 1)}}{{k^2 (k^2  - q^2  + 2)}}}\\
&= \sum_{k = 1}^q {\left( {\tan ^{ - 1} \frac{1}{{k^2  - qk + 1}}} \right)^2 }\,,
\end{split}
\end{equation}
giving, in particular,
\begin{equation}
\sum_{k = 1}^\infty  {\tan ^{ - 1} \frac{{2k}}{{k^4  + k^2  + 2}}\tan ^{ - 1} \frac{2}{{k^2 }}}  = \frac{{\pi ^2 }}{{16}}
\end{equation}
and
\begin{equation}
\sum_{k = 1}^\infty  {\tan ^{ - 1} \frac{{6k}}{{k^4  - 7k^2  + 2}}\tan ^{ - 1} \frac{{2(k^2  + 1)}}{{k^2 (k^2  - 7)}}}  = \frac{{3\pi ^2 }}{{16}}\,.
\end{equation}
With $m=1=q=\alpha=-\beta$ in the identity of Theorem~\ref{thm.lqwviov}, we have,
\begin{equation}
\sum_{k = 2}^\infty  {\tan ^{ - 1} \frac{{2k}}{{(k - 1)(k + 1)(k^2  - 2)}}\tan ^{ - 1} \frac{{2(k - 1)(k + 1)}}{{k^2 (k^2  - 3)}}}  = \frac{{\pi ^2 }}{{16}}\,.
\end{equation}
Putting $m=2=\alpha=\beta$ in the identity of Theorem~\ref{thm.lqwviov}, we have,
\begin{equation}
\begin{split}
&\sum_{k = 1}^\infty  {\tan ^{ - 1} \frac{{8qk}}{{k^4  - (q^2  - 1)4k^2  + 8}}\tan ^{ - 1} \frac{{4(k^2  + 2)}}{{k^2 (k^2  - 4q^2  + 4)}}\left( {\tan ^{ - 1} \frac{2}{{k^2  - q^2  + 2}}} \right)^2 }\\
&\qquad= \sum_{k = 1}^q {\left( {\tan ^{ - 1} \frac{2}{{k^2  - 2qk + 2}}} \right)^2 \left( {\tan ^{ - 1} \frac{2}{{k^2  - q^2  + 2}}} \right)^2 }\,,
\end{split}
\end{equation}
giving the special evaluation
\begin{equation}
\sum_{k = 1}^\infty  {\tan ^{ - 1} \frac{{8k}}{{k^4  + 8}}\tan ^{ - 1} \frac{{4(k^2  + 2)}}{{k^4 }}\left( {\tan ^{ - 1} \frac{2}{{k^2  + 1}}} \right)^2 }  = \frac{{\pi ^2 }}{{16}}(\tan ^{ - 1} 2)^2\,.
\end{equation}
Setting $m=2$ and $\alpha=1=-\beta$ gives
\begin{equation}
\begin{split}
&\sum_{k = 1}^\infty  {\tan ^{ - 1} \frac{{4qk}}{{k^4  - (2q^2  + 1)2k^2  + 2}}\tan ^{ - 1} \frac{{2(k - 1)(k + 1)}}{{k^2 (k^2  - 4q^2  - 2)}}}  \times\\
&\qquad\times \left( {\tan ^{ - 1} \frac{1}{{k^2  - q^2  - 1)}}} \right)^2\\
&\qquad\qquad= \sum_{k = 1}^q {\left( {\tan ^{ - 1} \frac{1}{{k^2  - 2kq - 1)}}} \right)^2 \left( {\tan ^{ - 1} \frac{1}{{k^2  - q^2  - 1)}}} \right)^2 }\,.
\end{split}
\end{equation}
\subsection{More examples}
Using $f(k)=(k-1)^3$ in Lemma~\ref{L1} at $m=1$ gives
\begin{equation}\label{equ.qz6skkb4}
\sum_{k = 1}^\infty  {\tan ^{ - 1} \frac{{\alpha (3k^2  - 3k + 1)}}{{k^3 (k - 1)^3  + \alpha ^2 }}}  = \frac{\pi }{2}
\end{equation}
at $q=1$ and
\begin{equation}\label{equ.ruu67fm}
\sum_{k = 1}^\infty  {\tan ^{ - 1} \frac{{2\alpha (3k^2  + 1)}}{{(k^2  - 1)^3  + \alpha ^2 }}}  = \frac{\pi }{2} + \tan ^{ - 1} \alpha\,,
\end{equation}
at $q=2$, for $\alpha$ a positive real number.
At $\alpha=1$, identity~\eqref{equ.qz6skkb4} gives
\begin{equation}
\sum_{k = 1}^\infty  {\tan ^{ - 1} \frac{{(3k^2  - 3k + 1)}}{{(k^2  - k + 1)(k^4  - 2k^3  + k + 1)}}}  = \frac{\pi }{2}
\end{equation}
while identity~\eqref{equ.ruu67fm} gives
\begin{equation}
\sum_{k = 1}^\infty  {\tan ^{ - 1} \frac{{2(3k^2  + 1)}}{{k^2 (k^4  - 3k^2  + 3)}}}  = \frac{{3\pi }}{4}\,.
\end{equation}

At $m=2$ and $q=1$, we have
\begin{equation}
\sum_{k = 1}^\infty  {\tan ^{ - 1} \frac{{2\alpha (3k^2  + 1)}}{{(k^2  - 1)^3  + \alpha ^2 }}\tan^{-1}\frac\alpha{k^3}}  = \frac{\pi }{2}\tan ^{ - 1} \alpha,\quad\alpha\in\R^+\,.
\end{equation}
Using $f(k)=k^2-3k+2$ and $m=1$ in Lemma~\ref{L1}, we have
\begin{equation}
\begin{split}
&\sum_{k = 1}^\infty  {\tan ^{ - 1} \frac{{\alpha q(2k + q - 3)}}{{(k - 1)(k - 2)(k + q - 1)(k + q - 2) + \alpha ^2 }}}\\
&\qquad\qquad\qquad  = \pi  + \sum_{k = 1}^{q - 2} {\tan ^{ - 1} \frac{\alpha }{{k(k + 1)}}}\,,
\end{split}
\end{equation}
from which we get
\begin{equation}
\sum_{k = 1}^\infty  {\tan ^{ - 1} \frac{{2\alpha (k - 1)}}{{(k - 1)^2 (k - 2)k + \alpha ^2 }}}  = \frac{\pi }{2}\,,
\end{equation}
\begin{equation}\label{equ.awkuyp9}
\sum_{k = 1}^\infty  {\tan ^{ - 1} \frac{{2\alpha (2k - 1)}}{{(k - 1)(k - 2)(k + 1)k + \alpha ^2 }}}  = \pi
\end{equation}
and
\begin{equation}\label{equ.kt8f9av}
\sum_{k = 1}^\infty  {\tan ^{ - 1} \frac{{6\alpha k}}{{(k - 1)(k - 2)(k + 2)(k + 1) + \alpha ^2 }}}  = \pi  + \tan ^{ - 1} \frac{\alpha }{2}\,,
\end{equation}
for positive real $\alpha$. 
Evaluating at $\alpha=1$, identity~\eqref{equ.awkuyp9} gives
\begin{equation}
\sum_{k = 1}^\infty  {\tan ^{ - 1} \frac{{2(2k - 1)}}{{(k^2  - k - 1)^2 }}}  = \pi\,,
\end{equation}
while at $\alpha=2$, identity~\eqref{equ.kt8f9av} reproduces identity~\eqref{equ.q0xt5i4}.

\bigskip

With $f(k)=k^2-3k+2$ and $m=2$, $q=1$ in Lemma~\ref{L1}, we have
\begin{equation}
\sum_{k = 1}^\infty  {\tan ^{ - 1} \frac{{2\alpha (2k + 1)}}{{(k - 1)k(k + 1)(k + 2) + \alpha ^2 }}\tan ^{ - 1} \frac{\alpha }{{k(k + 1)}}}  = \frac{\pi }{2}\tan ^{ - 1} \frac{\alpha }{2}\,,
\end{equation}
giving, in particular,
\begin{equation}
\sum_{k = 1}^\infty  {\tan ^{ - 1} \frac{{2(2k + 1)}}{{(k^2  + k - 1)^2 }}\tan ^{ - 1} \frac{1}{{k(k + 1)}}}  = \frac{\pi }{2}\tan ^{ - 1} \frac{1}{2}\,.
\end{equation}
Using $f(k)=(k^2-5k+5)^2$ in Lemma~\ref{L1} gives
\begin{equation}
\sum_{k = 1}^\infty  {\tan ^{ - 1} \frac{{8\alpha (2k - 1)(k^2  - k + 3)}}{{(k^2  + 3k + 1)^2 (k^2  - 5k + 5)^2  + \alpha ^2 }}\tan ^{ - 1} \frac{\alpha }{{(k^2  - k - 1)^2 }}}  = 2(\tan ^{ - 1} \alpha )^2\,,
\end{equation}
at $q=2$, $m=2$,
\begin{equation}
\begin{split}
&\sum_{k = 1}^\infty  {\tan ^{ - 1} \frac{{12\alpha (k - 1)(k^2  - 2k + 2)}}{{(k^2  + k - 1)^2 (k^2  - 5k + 5)^2  + \alpha ^2 }}\tan ^{ - 1} \frac{\alpha }{{(k^2  - 3k + 1)^2 }}}\times\\
&\qquad\qquad\times\tan ^{ - 1} \frac{\alpha }{{(k^2  - k - 1)^2 }} = (\tan ^{ - 1} \alpha )^3
\end{split}
\end{equation}
at $m=3$, $q=1$ and
\begin{equation}
\begin{split}
&\sum_{k = 1}^\infty  {\tan ^{ - 1} \frac{{8\alpha (2k - 1)(k^2  - k + 3)}}{{(k^2  + 3k + 1)^2 (k^2  - 5k + 5)^2  + \alpha ^2 }}\tan ^{ - 1} \frac{\alpha }{{(k^2  - 3k + 1)^2 }}}\times\\
&\qquad\qquad\times\tan ^{ - 1} \frac{\alpha }{{(k^2  - k - 1)^2 }}\tan ^{ - 1} \frac{\alpha }{{(k^2  + k - 1)^2 }} = (\tan ^{ - 1} \alpha )^4\,,
\end{split}
\end{equation}
at $m=4$, $q=1$, for real numbers $\alpha$.

\bigskip

Use of $f(k)=k^2-5k+5$ in Lemma~\ref{LP}, with $\alpha=1$, $m=1$ produces
\begin{equation}
\sum_{k = 1}^\infty  {\tan ^{ - 1} \frac{{2(k + 1)}}{{k(k + 2)(k^2  + 2k - 1)}}\tan ^{ - 1} \frac{{2k(k + 2)}}{{(k + 1)^2 (k^2  + 2k - 2)}}}  = \frac{{\pi ^2 }}{{16}}\,,
\end{equation}
at $q=1$ and
\begin{equation}
\sum_{k = 1}^\infty  {\tan ^{ - 1} \frac{{2(2k + 1)}}{{k(k + 1)(k^2  + k - 4)}}\tan ^{ - 1} \frac{{2k(k + 1)}}{{k^4  + 2k^3  - 3k^2  - 4k - 2}}}  = \frac{{\pi ^2 }}{8}\,,
\end{equation}
at $q=2$.

\bigskip

Use of $f(k)=k^2-5k+5$ in Lemma~\ref{LP}, with $q=1$, $\alpha=1$ and $m=3$ produces
\begin{equation}
\begin{split}
&\sum_{k = 1}^\infty  {\tan ^{ - 1} \frac{{6k}}{{k^4  - 7k^2  + 2}}\tan ^{ - 1} \frac{{2(k^2  + 1)}}{{k^2 (k^2  - 7)}}}\times\\ 
&\qquad\qquad\times\left( {\tan ^{ - 1} \frac{1}{{k^2  - k - 1}}} \right)^2 \left( {\tan ^{ - 1} \frac{1}{{k^2  + k - 1}}} \right)^2  = \frac{{\pi ^6 }}{{4096}}\,.
\end{split}
\end{equation}
\subsection{Alternating sums}
\begin{lemma}\label{L2}
Let $\alpha$ be a real number, let $m$ and $q$ be positive integers and let $\{ f(k)\} _{k = 1}^\infty  $ be a real positive non-decreasing sequence such that \mbox{$\lim_{k\to\infty}f(k)=\infty$}, then 
\[
\begin{split}
&\sum_{k = 1}^\infty  {(-1)^{k-1}\tan ^{ - 1} \frac{{\alpha (f(k + mq) \mp f(k))}}{{f(k)f(k + mq) \pm \alpha ^2 }}\prod_{j = 1}^{m - 1} {\tan ^{ - 1} \frac{\alpha }{{f(k + jq)}}} }\\
&\qquad\qquad= \sum_{k = 1}^q {(-1)^{k-1}\prod_{j = 0}^{m - 1} {\tan ^{ - 1} \frac{\alpha }{{f(k + jq)}}} }\,,
\end{split}
\]
where the upper signs apply if $q$ is even, and the lower if $q$ is odd.
\end{lemma}
\begin{lemma}\label{LPalt}
Let $\alpha$ be a real number, let $m$ and $q$ be positive integers and let $\{ f(k)\} _{k = 1}^\infty  $ be a real positive non-decreasing sequence such that \mbox{$\lim_{k\to\infty}f(k)=\infty$}, then 
\[
\begin{split}
&\sum_{k = 1}^\infty  {( - 1)^{k - 1} \tan ^{ - 1} \frac{{\alpha (f(k + 2mq) - f(k))}}{{f(k)f(k + 2mq) + \alpha ^2 }}\tan ^{ - 1} \frac{{\alpha (f(k + 2mq) + f(k))}}{{f(k)f(k + 2mq) - \alpha ^2 }} \times }\\
&\qquad\qquad\times \prod_{j = 1}^{m - 1} {\left( {\tan ^{ - 1} \frac{\alpha }{{f(k + 2jq)}}} \right)} ^2\\
&\qquad\qquad\qquad= \sum_{k = 1}^{2q} {( - 1)^{k - 1} \prod_{j = 0}^{m - 1} {\left( {\tan ^{ - 1} \frac{\alpha }{{f(k + 2jq)}}} \right)} ^2 }\,.
\end{split}
\]

\end{lemma}

Lemma~\ref{L2} and Lemma~\ref{LPalt} follow from identity~(2.4) of~\cite{adegoke} and the trigonometric identities~\eqref{equ.ie9kz2w} and \eqref{equ.zkvzqba}.

\bigskip

On account of Lemma~\ref{L2} with the upper signs and Lemma~\ref{LPalt}, alternating versions of the results obtained in the previous sections are readily obtained by replacing $q$ with $2q$. For example, corresponding to identity~\eqref{equ.s0rb2zk} is the following alternating sum
\begin{equation}
\sum_{k = 1}^\infty  {( - 1)^{k - 1} \tan ^{ - 1} \frac{{2\alpha q}}{{(k + q)^2  + \alpha ^2  - q^2 }}}  = \sum_{k = 1}^{2q} {( - 1)^{k - 1} \tan ^{ - 1} \frac{\alpha }{k}}\,,\quad\mbox{$\alpha\in\R$}\,,
\end{equation}
which gives, in particular,
\begin{equation}
\sum_{k = 1}^\infty  {( - 1)^{k - 1} \tan ^{ - 1} \frac{{2q^2 }}{{(k + q)^2 }}}  = \sum_{k = 1}^{2q} {( - 1)^{k - 1} \tan ^{ - 1} \frac{q}{k}}\,,
\end{equation}
from which we get the special value
\begin{equation}
\sum_{k = 1}^\infty  {( - 1)^{k-1} \tan ^{ - 1} \frac{2}{{k^2 }}}  = \frac{\pi }{4}\,.
\end{equation}
Similarly, the alternating version of identity~\eqref{equ.vs31gxz} is
\begin{equation}
\begin{split}
&\sum_{k = 1}^\infty  {( - 1)^{k - 1} \tan ^{ - 1} \frac{{8q^2 }}{{(k + 2q - 1)^2 }}\tan ^{ - 1} \frac{{2q}}{{(k + 2q - 1)}}}\\
&\qquad= \frac{{\pi ^2 }}{8} + \sum_{k = 2}^{2q} {( - 1)^{k - 1} \tan ^{ - 1} \frac{{2q}}{{k - 1}}\tan ^{ - 1} \frac{{2q}}{{(k + 2q - 1)}}}\,.
\end{split}
\end{equation}
The alternating version of identity~\eqref{equ.g3fwlld} is
\begin{equation}
\begin{split}
&\sum_{k = 1}^\infty  {( - 1)^{k - 1} \tan ^{ - 1} \frac{{4kq}}{{2k^4  - 2k^2 (4q^2  - 1)+1}}\tan ^{ - 1} \frac{{2k^2  + 1}}{{2k^2 (k^2  - 4q^2  + 1)}}}\\
&\qquad\qquad= \sum_{k = 1}^{2q} {( - 1)^{k - 1} \left( {\tan ^{ - 1} \frac{1}{{2k^2  - 4kq + 1}}} \right)^2 }\,,
\end{split}
\end{equation}
while that of identity~\eqref{equ.mhe9kvi} is
\begin{equation}
\begin{split}
&\sum_{k = 1}^\infty  {( - 1)^{k - 1} \tan ^{ - 1} \frac{{4kq}}{{2k^4  - 2k^2 (4q^2  + 1) + 1}}\tan ^{ - 1} \frac{{2k^2  - 1}}{{2k^2 (k^2  - 4q^2  - 1)}}}\\
&\qquad\qquad= \sum_{k = 1}^{2q} {( - 1)^{k - 1} \left( {\tan ^{ - 1} \frac{1}{{2k^2  - 4kq - 1}}} \right)^2 }\,.
\end{split}
\end{equation}
Additional alternating sums can be obtained from Lemma~\ref{L2} with the lower signs. For example, using $f(k)=k+\beta$ in Lemma~\ref{L2} with the lower signs we have the following result. 
\begin{thm}\label{thm.hy3p7rx}
If $\alpha$ and $\beta$ are real numbers and $m$ and $q$ are positive integers such that $q$ is odd, then
\[
\begin{split}
&\sum_{k = 1}^\infty  {(-1)^{k-1}\tan ^{ - 1} \frac{{\alpha (2k+2\beta+mq)}}{{k^2  + (2\beta + mq)k + \beta(\beta+mq) - \alpha ^2 }}\prod_{j = 1}^{m - 1} {\tan ^{ - 1} \frac{\alpha }{{k + jq + \beta}}} }\\
&\qquad\qquad\qquad= \sum_{k = 1}^q {(-1)^{k-1}\prod_{j = 0}^{m - 1} {\tan ^{ - 1} \frac{\alpha }{{k + jq + \beta}}} }\,.
\end{split}
\]

\end{thm}
\begin{cor}
If $m$ and $q$ are positive integers such that $q$ is odd, then
\[
\begin{split}
&\sum_{k = 1}^\infty  {( - 1)^{k - 1} \tan ^{ - 1} \frac{{m(2k + m(q - 1))}}{{2k^2  + 2mk(q - 1) - m^2 q}}\prod_{j = 1}^{m - 1} {\tan ^{ - 1} \frac{m}{{2k + 2jq - m}}} }\\
&\qquad\qquad= \sum_{k = 1}^q  {( - 1)^{k - 1} \prod_{j = 0}^{m - 1} {\tan ^{ - 1} \frac{m}{{2k + 2jq - m}}} }\,.
\end{split}
\]
\end{cor}
\begin{proof}
Set $\alpha=m/2=-\beta$ in Theorem~\ref{thm.hy3p7rx}.
\end{proof}
In particular, we have
\begin{equation}
\begin{split}
&\sum_{k = 1}^\infty  {( - 1)^{k - 1} \tan ^{ - 1} \frac{{2mk}}{{2k^2  - m^2 }}\prod_{j = 1}^{m - 1} {\tan ^{ - 1} \frac{m}{{2k + 2j - m}}} }\\
&\qquad\qquad= \prod_{j = 0}^{m - 1} {\tan ^{ - 1} \frac{m}{{2j - m + 2}}}\,,
\end{split}
\end{equation}
giving the special value
\begin{equation}
\sum_{k = 1}^\infty  {( - 1)^{k - 1} \tan ^{ - 1} \frac{{2k}}{{2k^2  - 1}}}  = \frac{\pi }{4}\,,
\end{equation}
and
\begin{equation}
\begin{split}
&\sum_{k = 1}^\infty  {( - 1)^k \tan ^{ - 1} \frac{{2(k + q - 1)}}{{k^2  + 2k(q - 1) - 2q}}\tan ^{ - 1} \frac{1}{{k + q - 1}}}\\
&= \frac{\pi }{2}\tan ^{ - 1} \frac{1}{q} + \sum_{k = 2}^q {( - 1)^k \tan ^{ - 1} \frac{1}{{k - 1}}\tan ^{ - 1} \frac{1}{{k + q - 1}}}\,,\quad\mbox{$q$ odd}\,,
\end{split}
\end{equation}
giving the special value,
\begin{equation}
\sum_{k = 1}^\infty  {( - 1)^k \tan ^{ - 1} \frac{{2k}}{{k^2  - 2}}\tan ^{ - 1} \frac{1}{k}}  = \frac{{\pi ^2 }}{8}\,.
\end{equation}
The next result is proved by using $f(k)=k^2-mqk+\beta$ in Lemma~\ref{L2} with the lower signs.
\begin{thm}\label{thm.stcc51n}
If $\alpha$ and $\beta$ are real numbers and $q$ and $m$ are positive integers such that $q$ is odd, then
\[
\begin{split}
&\sum_{k = 1}^\infty  {( - 1)^{k - 1} \tan ^{ - 1} \frac{{2\alpha (k^2  + \beta )}}{{k^4  - (m^2 q^2 - 2\beta )k^2  + \beta ^2  - \alpha ^2 }} \times }\\ 
&\qquad\qquad \times \prod_{j = 1}^{m - 1} {\tan ^{ - 1} \frac{\alpha }{{k^2  + (2j-m)qk - q^2 j(m - j)+\beta}}}\\ 
&\qquad\qquad\qquad = \sum_{k = 1}^q ( - 1)^{k - 1}{\prod_{j = 0}^{m - 1} {\tan ^{ - 1} \frac{\alpha }{{k^2  + (2j-m)qk - q^2 j(m - j)+\beta}}} }\,. 
\end{split}
\]

\end{thm}
Setting $\alpha=m^2/2=\beta$ and $q=1$ in the identity of Theorem~\ref{thm.stcc51n} gives
\begin{equation}
\begin{split}
&\sum_{k = 1}^\infty  {( - 1)^{k - 1} \tan ^{ - 1} \frac{{m^2 (2k^2  + m^2 )}}{{2k^4 }} \times }\\
&\qquad\qquad\qquad\times \prod_{j = 1}^{m - 1} {\tan ^{ - 1} \frac{{m^2 }}{{2k^2  + (2j - m)2k - 2j(m - j) + m^2 }}}\\
&\qquad\qquad\qquad\qquad = \prod_{j = 0}^{m - 1} {\tan ^{ - 1} \frac{{m^2 }}{{2k^2  + (2j - m)2k - 2j(m - j) + m^2 }}}\,,
\end{split}
\end{equation}
giving, in particular,
\begin{equation}
\sum_{k = 1}^\infty  {( - 1)^{k - 1} \tan ^{ - 1} \frac{{2k^2  + 1}}{{2k^4 }}}  = \frac{\pi }{4}
\end{equation}
and
\begin{equation}
\sum_{k = 1}^\infty  {( - 1)^{k - 1} \tan ^{ - 1} \frac{{4(k^2  + 2)}}{{k^4 }}\tan ^{ - 1} \frac{2}{{k^2  + 1}}}  = \frac{\pi }{4}\tan ^{ - 1} 2\,.
\end{equation}
Setting $\alpha=m=-\beta$ and $q=1$ in the identity of Theorem~\ref{thm.stcc51n} gives
\begin{equation}
\begin{split}
&\sum_{k = 1}^\infty  {( - 1)^{k - 1} \tan ^{ - 1} \frac{{2m(k^2  - m)}}{{k^2 (k^2  - 2m - m^2 )}} \times }\\ 
&\qquad\qquad\times \prod_{j = 1}^{m - 1} {\tan ^{ - 1} \frac{m}{{k^2  + (2j - m)k - 2j(m - j) + j^2  - m(j + 1)}}}\\ 
&\qquad\qquad\qquad= \prod_{j = 0}^{m - 1} {\tan ^{ - 1} \frac{m}{{1 - (m - j)(j + 2)}}}\,,
\end{split}
\end{equation}
which has
\begin{equation}
\sum_{k = 1}^\infty  {( - 1)^k \tan ^{ - 1} \frac{{2(k^2  - 1)}}{{k^2 (k^2  - 3)}}}  = \frac{\pi }{4}
\end{equation}
and
\begin{equation}
\sum_{k = 1}^\infty  {( - 1)^{k - 1} \tan ^{ - 1} \frac{{4(k^2  - 2)}}{{k^2 (k^2  - 8)}}\tan ^{ - 1} \frac{2}{{k^2  - 3}}}  = \frac{\pi }{4}\tan ^{ - 1} \frac{2}{3}\,,
\end{equation}
as special values.

\end{document}